%% file: main.tex
\documentclass[12pt]{amsart}
\usepackage{amsmath,amssymb,amscd,amsthm,graphicx,enumerate}
\usepackage{pstricks}
\usepackage[ocgcolorlinks,linktoc=all]{hyperref}
\hypersetup{citecolor=blue,linkcolor=red}
\usepackage{lipsum}

\numberwithin{equation}{section}
\theoremstyle{plain}

\makeatletter
\newtheorem*{rep@theorem}{\rep@title}
\newcommand{\newreptheorem}[2]{%
\newenvironment{rep#1}[1]{%
 \def\rep@title{#2 \ref{##1}}%
 \begin{rep@theorem}}%
 {\end{rep@theorem}}}
\makeatother

\newtheorem{theorem}[equation]{Theorem}
\newreptheorem{theorem}{Theorem}

\theoremstyle{remark}

\theoremstyle{definition}

\newcommand{\Lap}{\Delta}
\newcommand{\eps}{\varepsilon}

\def\XXint#1#2#3{{\setbox0=\hbox{$#1{#2#3}{\int}$}
     \vcenter{\hbox{$#2#3$}}\kern-.5\wd0}}

\begin{document}
\title[Allen-Cahn equation on three-spheres]{Low complexity solutions of the Allen-Cahn equation on three-spheres}
\author{Robert Haslhofer, Mohammad N. Ivaki}\thanks{R.H. has been partially supported by NSERC grant RGPIN-2016-04331 and a Connaught New Researcher Award. M.I. has been supported by a Marsden Postdoctoral Fellowship. Both authors thank the Fields Institute for providing an excellent research environment during the thematic program on Geometric Analysis.}

\date{}

\begin{abstract}
In this short note, we prove that on the three-sphere with any bumpy metric there exist at least four solutions of the Allen-Cahn equation with spherical interface and index at most two. The proof combines several recent results from the literature.
\end{abstract}
\maketitle
\section{Introduction}
In this short note, we consider the Allen-Cahn equation,
\begin{equation}\label{eqAC}
-\Lap_g u = \frac{1}{\eps^2}\left(u-u^3\right),
\end{equation}
on $S^3$ endowed with any bumpy metric $g$. The Allen-Cahn equation has its origin as a model for phase transitions \cite{AC79} and has received a lot of attention especially due to its close connections with the theory of minimal surfaces, see e.g., the  surveys \cite{T08,S09,P12}.

In a beautiful recent paper \cite{GG}, Gaspar-Guaraco proved that on any closed manifold the number of solutions of \eqref{eqAC} goes to infinity as $\eps\to 0$. Most of these solutions naturally have high complexity, i.e., large index and interfaces of high genus. In this note, we address the remaining question of finding solutions with low complexity, i.e., solutions with interface of genus zero and low index, and prove:
\begin{theorem}\label{main_thm}
Consider the Allen-Cahn equation \eqref{eqAC} on $S^3$ endowed with any bumpy metric $g$.
Then for any small enough $\eps>0$ there exist at least four solutions with spherical interface and index at most two. More precisely, exactly one of the following alternatives holds:
\begin{enumerate}
\item There are at least two stable solutions with spherical interface, and at least four index-one solutions with spherical interface.
\item There are no stable solutions with spherical interface, at least two index-one solutions with spherical interface, and at least two index-two solutions with spherical interface.
\end{enumerate}
\end{theorem}
The two cases are illustrated in Figure 1 and Figure 2 for the scenarios that $(S^3,g)$ looks like a dumbbell and an elongated ellipsoid, respectively.
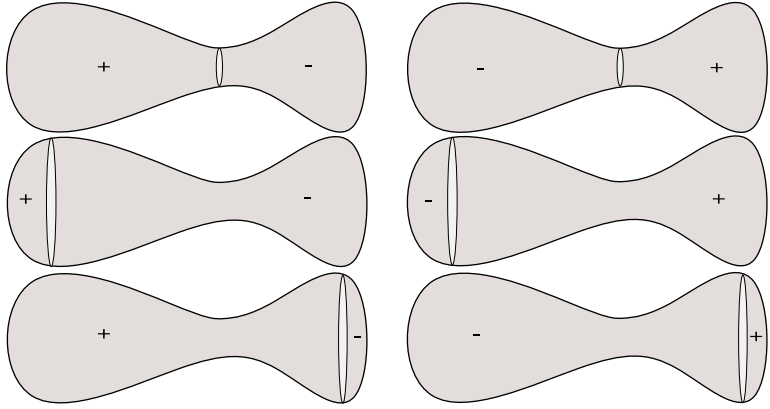
\begin{figure}[h]
\centering
\input{figure2.tex}
\caption{Solutions on a dumbbell}
\end{figure}
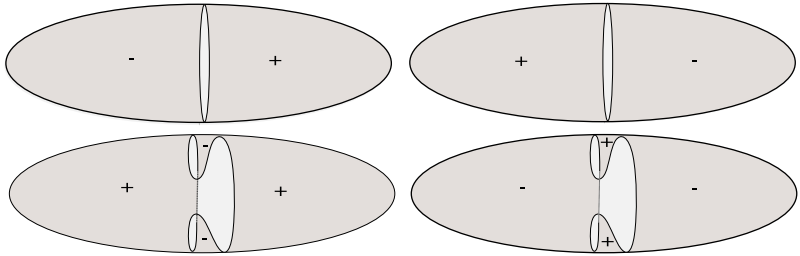
\begin{figure}[h]
\centering
\input{figure1.tex}
\caption{Solutions on elongated ellipsoids}
\end{figure}
\newline
The solutions transition from $u_\eps\approx -1$ to $u_\eps \approx +1$ along the spherical interface, and the transition is modeled on the function $\tanh(\frac{t}{\sqrt{2}\eps})$, where $t$ denotes the signed distance from the interface.

We recall that solutions of \eqref{eqAC} are critical points of the energy functional
\begin{equation}
E_\eps[u]=\int_M \frac{\eps}{2}|\nabla u|^2 + \frac{1}{4\eps} (1-u^2)^2\operatorname{dvol}_g
\end{equation}
The index of a critical point is defined as the number of directions in which the second variation of the functional is negative.

We also recall that a Riemannian metric $g$ on a manifold $M$ is called bumpy if $(M,g)$ does not contain any immersed
closed minimal hypersurface with a non-trivial Jacobi field. By a result of White \cite{W91,W17} bumpy metrics are generic in the sense of Baire.

To prove Theorem \ref{main_thm}, we combine several results from the literature, including in particular the gluing construction from Pacard-Ritore \cite{PR03}, the existence result for minimal two-spheres from Haslhofer-Ketover \cite{HK}, and the index estimate from Chodosh-Mantoulidis \cite{CM}. An interesting feature is that the proof relies on some methods from other fields, in particular, mean curvature flow with surgery, that have not been used before in the analysis of the Allen-Cahn equation.
\section{The proof}
\begin{proof}[Proof of Theorem \ref{main_thm}]
We first consider the case that there exists at least one solution $u_\eps$ of \eqref{eqAC} which is stable and has spherical interface (for every small enough $\eps$). Having a spherical interface, means that for $\eps\to 0$ (after passing to a subsequence), the associated varifolds
\begin{equation}
V_\eps(\phi)=\frac{3}{4}\int \eps |\nabla u_\eps (x)|^2 \phi (x,T_x\{ u_\eps = u_\eps(x)\}) \operatorname{dvol}_g(x)
\end{equation}
converge to an embedded minimal two-sphere $\Sigma$, possibly with multiplicity. It follows from classical theory \cite{T95,TW12} (see also \cite{H}) that $\Sigma$ is a stable critical point of the area functional. Thus, by a result of Ketover-Liokumovich \cite[Thm. 1.6]{KL} each of the two balls $B^{\ell/r}$ in $S^3$ bounded by $\Sigma$ contains in its interior an index-one minimal two-sphere $\Sigma^{\ell/r}$. We can now apply the main theorem of Pacard-Ritore \cite{PR03} (see also \cite{P12}) to get solutions $u_\eps^{\ell/r}$ with spherical interface $\Sigma^{\ell/r}$. Since for the Pacard-Ritore solutions $u_\eps^{\ell/r}$ the associated varifolds $V_\eps^{\ell/r}$ converge to $\Sigma^{\ell/r}$ with multiplicity one, it follows from Chodosh-Mantoulidis \cite[Sec. 6]{CM} (see also \cite[Sec. 9]{dPKW13} for closely related computations) that the index is actually continuous under the limit, i.e., for $\eps$ small enough we get
\begin{equation}
\textrm{ind}(u_\eps^{\ell/r})=\textrm{ind}(\Sigma^{\ell/r})=1,
\end{equation}
where the possibility of nullity is ruled out by the bumpyness assumption. Together with the obvious fact that whenever $u$ is a solution of \eqref{eqAC} then $-u$ also is a solution, this finishes the proof in case (1).\\

Consider now the case that \eqref{eqAC} admits no stable solutions with spherical interface. Suppose towards a contradiction that $(S^3,g)$ contains a stable embedded minimal two-sphere $\Sigma$. Then, arguing as above, for $\eps$ small enough the associated Pacard-Ritore solutions would satisfy
\begin{equation}
\textrm{ind}(u_\eps)=\textrm{ind}(\Sigma)=0,
\end{equation}
contradicting our assumption that \eqref{eqAC} has no stable solutions with spherical interface. Hence, by a classical result of Simon-Smith \cite{SS82} (see also \cite[Lem. 3.5]{MN12} and \cite[Lem. 3.2]{HK}) there exists an embedded minimal two-sphere $\Sigma^1\subset (S^3,g)$, which is obtained by one parameter min-max, realizes the 1-width, and has index one (since it can't be stable). By a recent result from Haslhofer-Ketover \cite[Thm. 1.4]{HK}, which has been established using combined efforts from mean curvature flow with surgery \cite{BH,HK17,BHH} and min-max theory \cite{KMN}, there also exists an embedded minimal two-sphere $\Sigma^2\subset (S^3,g)$ which has index two. Arguing as above, for $\eps$ small enough the associated Pacard-Ritore solutions satisfy
\begin{equation}
\textrm{ind}(\pm u^1_\eps)=\textrm{ind}(\Sigma^1)=1,
\end{equation}
and
\begin{equation}
 \textrm{ind}(\pm u^2_\eps)=\textrm{ind}(\Sigma^2)=2.
\end{equation}
This establishes case (2), and thus finishes the proof of the theorem.
\end{proof}

\vspace{10mm}
{\sc Department of Mathematics, University of Toronto, 40 St George Street, Toronto, ON M5S 2E4, Canada}

\emph{E-mail: roberth@math.toronto.edu, m.ivaki@utoronto.ca}

\end{document}

%% file: figure2.tex
\psset{xunit=.5pt,yunit=.5pt,runit=.5pt}
\begin{pspicture}(576.75356407,304.3213179)
{
\newrgbcolor{curcolor}{0.89019608 0.87058824 0.85882354}
\pscustom[linestyle=none,fillstyle=solid,fillcolor=curcolor]
{
\newpath
\moveto(427.95263428,280.31185326)
\curveto(400.23380623,291.59150042)(348.98714029,313.85465183)(323.32529418,297.42554697)
\curveto(297.66344807,280.9964421)(297.58638766,225.87504769)(323.13268376,210.45178356)
\curveto(348.67897987,195.02851943)(399.84852765,219.30337741)(434.17991241,231.44080255)
\curveto(468.51129718,243.57822769)(486.00453139,243.57822769)(507.38935437,231.44082367)
\curveto(528.77417735,219.30341964)(554.05076263,195.02850666)(566.6890459,210.3176586)
\curveto(579.32732916,225.60681055)(579.32732916,280.45997025)(565.79548169,297.38455673)
\curveto(552.26363422,314.30914322)(525.20002602,293.30525709)(505.98759903,282.15964548)
\curveto(486.77517204,271.01403386)(475.41429949,269.72698252)(467.63834641,269.2315192)
\curveto(459.86239333,268.73605589)(455.67146232,269.03220611)(427.95263428,280.31185326)
\closepath
}
}
{
\newrgbcolor{curcolor}{0 0 0}
\pscustom[linewidth=1.00000001,linecolor=curcolor]
{
\newpath
\moveto(427.95263428,280.31185326)
\curveto(400.23380623,291.59150042)(348.98714029,313.85465183)(323.32529418,297.42554697)
\curveto(297.66344807,280.9964421)(297.58638766,225.87504769)(323.13268376,210.45178356)
\curveto(348.67897987,195.02851943)(399.84852765,219.30337741)(434.17991241,231.44080255)
\curveto(468.51129718,243.57822769)(486.00453139,243.57822769)(507.38935437,231.44082367)
\curveto(528.77417735,219.30341964)(554.05076263,195.02850666)(566.6890459,210.3176586)
\curveto(579.32732916,225.60681055)(579.32732916,280.45997025)(565.79548169,297.38455673)
\curveto(552.26363422,314.30914322)(525.20002602,293.30525709)(505.98759903,282.15964548)
\curveto(486.77517204,271.01403386)(475.41429949,269.72698252)(467.63834641,269.2315192)
\curveto(459.86239333,268.73605589)(455.67146232,269.03220611)(427.95263428,280.31185326)
\closepath
}
}
{
\newrgbcolor{curcolor}{0.94901961 0.94901961 0.94901961}
\pscustom[linestyle=none,fillstyle=solid,fillcolor=curcolor]
{
\newpath
\moveto(467.18051057,254.87439182)
\curveto(467.18051057,247.02353312)(466.13354938,240.65915465)(464.84205729,240.65915465)
\curveto(463.55056521,240.65915465)(462.50360402,247.02353312)(462.50360402,254.87439182)
\curveto(462.50360402,262.72525052)(463.55056521,269.08962899)(464.84205729,269.08962899)
\curveto(466.13354938,269.08962899)(467.18051057,262.72525052)(467.18051057,254.87439182)
\closepath
}
}
{
\newrgbcolor{curcolor}{0 0 0}
\pscustom[linewidth=0.75590547,linecolor=curcolor]
{
\newpath
\moveto(467.18051057,254.87439182)
\curveto(467.18051057,247.02353312)(466.13354938,240.65915465)(464.84205729,240.65915465)
\curveto(463.55056521,240.65915465)(462.50360402,247.02353312)(462.50360402,254.87439182)
\curveto(462.50360402,262.72525052)(463.55056521,269.08962899)(464.84205729,269.08962899)
\curveto(466.13354938,269.08962899)(467.18051057,262.72525052)(467.18051057,254.87439182)
\closepath
}
}
{
\newrgbcolor{curcolor}{0.89019608 0.87058824 0.85882354}
\pscustom[linestyle=none,fillstyle=solid,fillcolor=curcolor]
{
\newpath
\moveto(124.62140782,280.53353389)
\curveto(96.90257977,291.81318105)(45.65591383,314.07633246)(19.99406772,297.6472276)
\curveto(-5.66777839,281.21812273)(-5.7448388,226.09672832)(19.80145731,210.67346419)
\curveto(45.34775341,195.25020006)(96.51730119,219.52505804)(130.84868596,231.66248318)
\curveto(165.18007072,243.79990832)(182.67330494,243.79990832)(204.05812792,231.6625043)
\curveto(225.4429509,219.52510027)(250.71953618,195.25018729)(263.35781944,210.53933923)
\curveto(275.9961027,225.82849118)(275.9961027,280.68165088)(262.46425523,297.60623736)
\curveto(248.93240777,314.53082385)(221.86879957,293.52693772)(202.65637258,282.38132611)
\curveto(183.44394559,271.23571449)(172.08307304,269.94866315)(164.30711996,269.45319983)
\curveto(156.53116688,268.95773652)(152.34023587,269.25388674)(124.62140782,280.53353389)
\closepath
}
}
{
\newrgbcolor{curcolor}{0 0 0}
\pscustom[linewidth=1.00000001,linecolor=curcolor]
{
\newpath
\moveto(124.62140782,280.53353389)
\curveto(96.90257977,291.81318105)(45.65591383,314.07633246)(19.99406772,297.6472276)
\curveto(-5.66777839,281.21812273)(-5.7448388,226.09672832)(19.80145731,210.67346419)
\curveto(45.34775341,195.25020006)(96.51730119,219.52505804)(130.84868596,231.66248318)
\curveto(165.18007072,243.79990832)(182.67330494,243.79990832)(204.05812792,231.6625043)
\curveto(225.4429509,219.52510027)(250.71953618,195.25018729)(263.35781944,210.53933923)
\curveto(275.9961027,225.82849118)(275.9961027,280.68165088)(262.46425523,297.60623736)
\curveto(248.93240777,314.53082385)(221.86879957,293.52693772)(202.65637258,282.38132611)
\curveto(183.44394559,271.23571449)(172.08307304,269.94866315)(164.30711996,269.45319983)
\curveto(156.53116688,268.95773652)(152.34023587,269.25388674)(124.62140782,280.53353389)
\closepath
}
}
{
\newrgbcolor{curcolor}{0.89019608 0.87058824 0.85882354}
\pscustom[linestyle=none,fillstyle=solid,fillcolor=curcolor]
{
\newpath
\moveto(427.76849569,75.81472445)
\curveto(400.04966764,87.0943716)(348.80300171,109.35752301)(323.1411556,92.92841815)
\curveto(297.47930949,76.49931328)(297.40224907,21.37791887)(322.94854518,5.95465474)
\curveto(348.49484129,-9.46860939)(399.66438907,14.8062486)(433.99577383,26.94367373)
\curveto(468.3271586,39.08109887)(485.82039281,39.08109887)(507.20521579,26.94369485)
\curveto(528.59003877,14.80629083)(553.86662405,-9.46862216)(566.50490731,5.82052979)
\curveto(579.14319058,21.10968173)(579.14319058,75.96284143)(565.61134311,92.88742791)
\curveto(552.07949564,109.8120144)(525.01588744,88.80812827)(505.80346045,77.66251666)
\curveto(486.59103346,66.51690504)(475.23016091,65.2298537)(467.45420783,64.73439039)
\curveto(459.67825475,64.23892708)(455.48732374,64.53507729)(427.76849569,75.81472445)
\closepath
}
}
{
\newrgbcolor{curcolor}{0 0 0}
\pscustom[linewidth=1.00000001,linecolor=curcolor]
{
\newpath
\moveto(427.76849569,75.81472445)
\curveto(400.04966764,87.0943716)(348.80300171,109.35752301)(323.1411556,92.92841815)
\curveto(297.47930949,76.49931328)(297.40224907,21.37791887)(322.94854518,5.95465474)
\curveto(348.49484129,-9.46860939)(399.66438907,14.8062486)(433.99577383,26.94367373)
\curveto(468.3271586,39.08109887)(485.82039281,39.08109887)(507.20521579,26.94369485)
\curveto(528.59003877,14.80629083)(553.86662405,-9.46862216)(566.50490731,5.82052979)
\curveto(579.14319058,21.10968173)(579.14319058,75.96284143)(565.61134311,92.88742791)
\curveto(552.07949564,109.8120144)(525.01588744,88.80812827)(505.80346045,77.66251666)
\curveto(486.59103346,66.51690504)(475.23016091,65.2298537)(467.45420783,64.73439039)
\curveto(459.67825475,64.23892708)(455.48732374,64.53507729)(427.76849569,75.81472445)
\closepath
}
}
{
\newrgbcolor{curcolor}{0.94901961 0.94901961 0.94901961}
\pscustom[linestyle=none,fillstyle=solid,fillcolor=curcolor]
{
\newpath
\moveto(560.92272589,49.09454755)
\curveto(560.92272589,22.37555637)(559.46687572,0.71553437)(557.67099353,0.71553437)
\curveto(555.87511134,0.71553437)(554.41926118,22.37555637)(554.41926118,49.09454755)
\curveto(554.41926118,75.81353873)(555.87511134,97.47356073)(557.67099353,97.47356073)
\curveto(559.46687572,97.47356073)(560.92272589,75.81353873)(560.92272589,49.09454755)
\closepath
}
}
{
\newrgbcolor{curcolor}{0 0 0}
\pscustom[linewidth=0.75590552,linecolor=curcolor]
{
\newpath
\moveto(560.92272589,49.09454755)
\curveto(560.92272589,22.37555637)(559.46687572,0.71553437)(557.67099353,0.71553437)
\curveto(555.87511134,0.71553437)(554.41926118,22.37555637)(554.41926118,49.09454755)
\curveto(554.41926118,75.81353873)(555.87511134,97.47356073)(557.67099353,97.47356073)
\curveto(559.46687572,97.47356073)(560.92272589,75.81353873)(560.92272589,49.09454755)
\closepath
}
}
{
\newrgbcolor{curcolor}{0.89019608 0.87058824 0.85882354}
\pscustom[linestyle=none,fillstyle=solid,fillcolor=curcolor]
{
\newpath
\moveto(125.13129632,75.47715216)
\curveto(97.41246827,86.75679932)(46.16580234,109.01995072)(20.50395623,92.59084586)
\curveto(-5.15788988,76.161741)(-5.2349503,21.04034659)(20.31134581,5.61708246)
\curveto(45.85764192,-9.80618167)(97.0271897,14.46867631)(131.35857446,26.60610145)
\curveto(165.68995923,38.74352659)(183.18319344,38.74352659)(204.56801642,26.60612257)
\curveto(225.9528394,14.46871854)(251.22942468,-9.80619445)(263.86770794,5.4829575)
\curveto(276.50599121,20.77210945)(276.50599121,75.62526915)(262.97414374,92.54985563)
\curveto(249.44229627,109.47444211)(222.37868807,88.47055599)(203.16626108,77.32494437)
\curveto(183.95383409,66.17933276)(172.59296154,64.89228141)(164.81700846,64.3968181)
\curveto(157.04105538,63.90135479)(152.85012437,64.19750501)(125.13129632,75.47715216)
\closepath
}
}
{
\newrgbcolor{curcolor}{0 0 0}
\pscustom[linewidth=1.00000001,linecolor=curcolor]
{
\newpath
\moveto(125.13129632,75.47715216)
\curveto(97.41246827,86.75679932)(46.16580234,109.01995072)(20.50395623,92.59084586)
\curveto(-5.15788988,76.161741)(-5.2349503,21.04034659)(20.31134581,5.61708246)
\curveto(45.85764192,-9.80618167)(97.0271897,14.46867631)(131.35857446,26.60610145)
\curveto(165.68995923,38.74352659)(183.18319344,38.74352659)(204.56801642,26.60612257)
\curveto(225.9528394,14.46871854)(251.22942468,-9.80619445)(263.86770794,5.4829575)
\curveto(276.50599121,20.77210945)(276.50599121,75.62526915)(262.97414374,92.54985563)
\curveto(249.44229627,109.47444211)(222.37868807,88.47055599)(203.16626108,77.32494437)
\curveto(183.95383409,66.17933276)(172.59296154,64.89228141)(164.81700846,64.3968181)
\curveto(157.04105538,63.90135479)(152.85012437,64.19750501)(125.13129632,75.47715216)
\closepath
}
}
{
\newrgbcolor{curcolor}{0.94901961 0.94901961 0.94901961}
\pscustom[linestyle=none,fillstyle=solid,fillcolor=curcolor]
{
\newpath
\moveto(258.28551693,48.75697027)
\curveto(258.28551693,22.03797909)(256.82966677,0.37795709)(255.03378458,0.37795709)
\curveto(253.23790239,0.37795709)(251.78205223,22.03797909)(251.78205223,48.75697027)
\curveto(251.78205223,75.47596145)(253.23790239,97.13598345)(255.03378458,97.13598345)
\curveto(256.82966677,97.13598345)(258.28551693,75.47596145)(258.28551693,48.75697027)
\closepath
}
}
{
\newrgbcolor{curcolor}{0 0 0}
\pscustom[linewidth=0.75590552,linecolor=curcolor]
{
\newpath
\moveto(258.28551693,48.75697027)
\curveto(258.28551693,22.03797909)(256.82966677,0.37795709)(255.03378458,0.37795709)
\curveto(253.23790239,0.37795709)(251.78205223,22.03797909)(251.78205223,48.75697027)
\curveto(251.78205223,75.47596145)(253.23790239,97.13598345)(255.03378458,97.13598345)
\curveto(256.82966677,97.13598345)(258.28551693,75.47596145)(258.28551693,48.75697027)
\closepath
}
}
{
\newrgbcolor{curcolor}{0 0 0}
\pscustom[linestyle=none,fillstyle=solid,fillcolor=curcolor]
{
\newpath
\moveto(268.34735305,50.26652129)
\lineto(263.23172296,50.26652129)
\lineto(263.23172296,51.19680506)
\lineto(268.34735305,51.19680506)
\closepath
}
}
{
\newrgbcolor{curcolor}{0 0 0}
\pscustom[linestyle=none,fillstyle=solid,fillcolor=curcolor]
{
\newpath
\moveto(358.87164571,51.59135505)
\lineto(353.75601562,51.59135505)
\lineto(353.75601562,52.52163882)
\lineto(358.87164571,52.52163882)
\closepath
}
}
{
\newrgbcolor{curcolor}{0.89019608 0.87058824 0.85882354}
\pscustom[linestyle=none,fillstyle=solid,fillcolor=curcolor]
{
\newpath
\moveto(125.13129632,178.74443029)
\curveto(97.41246827,190.02407744)(46.16580234,212.28722885)(20.50395623,195.85812399)
\curveto(-5.15788988,179.42901912)(-5.2349503,124.30762471)(20.31134581,108.88436058)
\curveto(45.85764192,93.46109645)(97.0271897,117.73595443)(131.35857446,129.87337957)
\curveto(165.68995923,142.01080471)(183.18319344,142.01080471)(204.56801642,129.87340069)
\curveto(225.9528394,117.73599667)(251.22942468,93.46108368)(263.86770794,108.75023562)
\curveto(276.50599121,124.03938757)(276.50599121,178.89254727)(262.97414374,195.81713375)
\curveto(249.44229627,212.74172024)(222.37868807,191.73783411)(203.16626108,180.5922225)
\curveto(183.95383409,169.44661088)(172.59296154,168.15955954)(164.81700846,167.66409623)
\curveto(157.04105538,167.16863292)(152.85012437,167.46478313)(125.13129632,178.74443029)
\closepath
}
}
{
\newrgbcolor{curcolor}{0 0 0}
\pscustom[linewidth=1.00000001,linecolor=curcolor]
{
\newpath
\moveto(125.13129632,178.74443029)
\curveto(97.41246827,190.02407744)(46.16580234,212.28722885)(20.50395623,195.85812399)
\curveto(-5.15788988,179.42901912)(-5.2349503,124.30762471)(20.31134581,108.88436058)
\curveto(45.85764192,93.46109645)(97.0271897,117.73595443)(131.35857446,129.87337957)
\curveto(165.68995923,142.01080471)(183.18319344,142.01080471)(204.56801642,129.87340069)
\curveto(225.9528394,117.73599667)(251.22942468,93.46108368)(263.86770794,108.75023562)
\curveto(276.50599121,124.03938757)(276.50599121,178.89254727)(262.97414374,195.81713375)
\curveto(249.44229627,212.74172024)(222.37868807,191.73783411)(203.16626108,180.5922225)
\curveto(183.95383409,169.44661088)(172.59296154,168.15955954)(164.81700846,167.66409623)
\curveto(157.04105538,167.16863292)(152.85012437,167.46478313)(125.13129632,178.74443029)
\closepath
}
}
{
\newrgbcolor{curcolor}{0.94901961 0.94901961 0.94901961}
\pscustom[linestyle=none,fillstyle=solid,fillcolor=curcolor]
{
\newpath
\moveto(37.99923787,152.51789412)
\curveto(37.99923787,125.61234472)(36.38970814,103.8010874)(34.4042526,103.8010874)
\curveto(32.41879706,103.8010874)(30.80926733,125.61234472)(30.80926733,152.51789412)
\curveto(30.80926733,179.42344353)(32.41879706,201.23470085)(34.4042526,201.23470085)
\curveto(36.38970814,201.23470085)(37.99923787,179.42344353)(37.99923787,152.51789412)
\closepath
}
}
{
\newrgbcolor{curcolor}{0 0 0}
\pscustom[linewidth=0.75590552,linecolor=curcolor]
{
\newpath
\moveto(37.99923787,152.51789412)
\curveto(37.99923787,125.61234472)(36.38970814,103.8010874)(34.4042526,103.8010874)
\curveto(32.41879706,103.8010874)(30.80926733,125.61234472)(30.80926733,152.51789412)
\curveto(30.80926733,179.42344353)(32.41879706,201.23470085)(34.4042526,201.23470085)
\curveto(36.38970814,201.23470085)(37.99923787,179.42344353)(37.99923787,152.51789412)
\closepath
}
}
{
\newrgbcolor{curcolor}{0.89019608 0.87058824 0.85882354}
\pscustom[linestyle=none,fillstyle=solid,fillcolor=curcolor]
{
\newpath
\moveto(427.77333349,179.27060855)
\curveto(400.05450544,190.55025571)(348.8078395,212.81340712)(323.14599339,196.38430225)
\curveto(297.48414728,179.95519739)(297.40708687,124.83380298)(322.95338298,109.41053885)
\curveto(348.49967908,93.98727472)(399.66922686,118.2621327)(434.00061163,130.39955784)
\curveto(468.33199639,142.53698298)(485.82523061,142.53698298)(507.21005359,130.39957896)
\curveto(528.59487657,118.26217493)(553.87146185,93.98726194)(566.50974511,109.27641389)
\curveto(579.14802837,124.56556584)(579.14802837,179.41872554)(565.6161809,196.34331202)
\curveto(552.08433344,213.2678985)(525.02072524,192.26401238)(505.80829825,181.11840076)
\curveto(486.59587126,169.97278915)(475.2349987,168.6857378)(467.45904563,168.19027449)
\curveto(459.68309255,167.69481118)(455.49216154,167.9909614)(427.77333349,179.27060855)
\closepath
}
}
{
\newrgbcolor{curcolor}{0 0 0}
\pscustom[linewidth=1.00000001,linecolor=curcolor]
{
\newpath
\moveto(427.77333349,179.27060855)
\curveto(400.05450544,190.55025571)(348.8078395,212.81340712)(323.14599339,196.38430225)
\curveto(297.48414728,179.95519739)(297.40708687,124.83380298)(322.95338298,109.41053885)
\curveto(348.49967908,93.98727472)(399.66922686,118.2621327)(434.00061163,130.39955784)
\curveto(468.33199639,142.53698298)(485.82523061,142.53698298)(507.21005359,130.39957896)
\curveto(528.59487657,118.26217493)(553.87146185,93.98726194)(566.50974511,109.27641389)
\curveto(579.14802837,124.56556584)(579.14802837,179.41872554)(565.6161809,196.34331202)
\curveto(552.08433344,213.2678985)(525.02072524,192.26401238)(505.80829825,181.11840076)
\curveto(486.59587126,169.97278915)(475.2349987,168.6857378)(467.45904563,168.19027449)
\curveto(459.68309255,167.69481118)(455.49216154,167.9909614)(427.77333349,179.27060855)
\closepath
}
}
{
\newrgbcolor{curcolor}{0 0 0}
\pscustom[linestyle=none,fillstyle=solid,fillcolor=curcolor]
{
\newpath
\moveto(322.35696599,153.04006667)
\lineto(317.24133588,153.04006667)
\lineto(317.24133588,153.97035044)
\lineto(322.35696599,153.97035044)
\closepath
}
}
{
\newrgbcolor{curcolor}{0.94901961 0.94901961 0.94901961}
\pscustom[linestyle=none,fillstyle=solid,fillcolor=curcolor]
{
\newpath
\moveto(163.96768723,254.94602643)
\curveto(163.96768723,247.09516772)(162.92072603,240.73078925)(161.62923395,240.73078925)
\curveto(160.33774187,240.73078925)(159.29078067,247.09516772)(159.29078067,254.94602643)
\curveto(159.29078067,262.79688513)(160.33774187,269.1612636)(161.62923395,269.1612636)
\curveto(162.92072603,269.1612636)(163.96768723,262.79688513)(163.96768723,254.94602643)
\closepath
}
}
{
\newrgbcolor{curcolor}{0 0 0}
\pscustom[linewidth=0.75590547,linecolor=curcolor]
{
\newpath
\moveto(163.96768723,254.94602643)
\curveto(163.96768723,247.09516772)(162.92072603,240.73078925)(161.62923395,240.73078925)
\curveto(160.33774187,240.73078925)(159.29078067,247.09516772)(159.29078067,254.94602643)
\curveto(159.29078067,262.79688513)(160.33774187,269.1612636)(161.62923395,269.1612636)
\curveto(162.92072603,269.1612636)(163.96768723,262.79688513)(163.96768723,254.94602643)
\closepath
}
}
{
\newrgbcolor{curcolor}{0 0 0}
\pscustom[linestyle=none,fillstyle=solid,fillcolor=curcolor]
{
\newpath
\moveto(78.77707986,52.46236808)
\lineto(74.84919161,52.46236808)
\lineto(74.84919161,49.375435)
\lineto(73.6528073,49.375435)
\lineto(73.6528073,52.46236808)
\lineto(69.72491905,52.46236808)
\lineto(69.72491905,53.36862367)
\lineto(73.6528073,53.36862367)
\lineto(73.6528073,56.45555676)
\lineto(74.84919161,56.45555676)
\lineto(74.84919161,53.36862367)
\lineto(78.77707986,53.36862367)
\closepath
}
}
{
\newrgbcolor{curcolor}{0 0 0}
\pscustom[linestyle=none,fillstyle=solid,fillcolor=curcolor]
{
\newpath
\moveto(542.43712764,253.98781636)
\lineto(538.50923939,253.98781636)
\lineto(538.50923939,250.90088327)
\lineto(537.31285508,250.90088327)
\lineto(537.31285508,253.98781636)
\lineto(533.38496683,253.98781636)
\lineto(533.38496683,254.89407194)
\lineto(537.31285508,254.89407194)
\lineto(537.31285508,257.98100503)
\lineto(538.50923939,257.98100503)
\lineto(538.50923939,254.89407194)
\lineto(542.43712764,254.89407194)
\closepath
}
}
{
\newrgbcolor{curcolor}{0 0 0}
\pscustom[linestyle=none,fillstyle=solid,fillcolor=curcolor]
{
\newpath
\moveto(572.23664167,50.44204836)
\lineto(568.30875342,50.44204836)
\lineto(568.30875342,47.35511527)
\lineto(567.11236911,47.35511527)
\lineto(567.11236911,50.44204836)
\lineto(563.18448086,50.44204836)
\lineto(563.18448086,51.34830395)
\lineto(567.11236911,51.34830395)
\lineto(567.11236911,54.43523703)
\lineto(568.30875342,54.43523703)
\lineto(568.30875342,51.34830395)
\lineto(572.23664167,51.34830395)
\closepath
}
}
{
\newrgbcolor{curcolor}{0 0 0}
\pscustom[linestyle=none,fillstyle=solid,fillcolor=curcolor]
{
\newpath
\moveto(19.68317931,154.48778566)
\lineto(15.75529106,154.48778566)
\lineto(15.75529106,151.40085257)
\lineto(14.55890675,151.40085257)
\lineto(14.55890675,154.48778566)
\lineto(10.6310185,154.48778566)
\lineto(10.6310185,155.39404124)
\lineto(14.55890675,155.39404124)
\lineto(14.55890675,158.48097433)
\lineto(15.75529106,158.48097433)
\lineto(15.75529106,155.39404124)
\lineto(19.68317931,155.39404124)
\closepath
}
}
{
\newrgbcolor{curcolor}{0.94901961 0.94901961 0.94901961}
\pscustom[linestyle=none,fillstyle=solid,fillcolor=curcolor]
{
\newpath
\moveto(341.49581702,153.42569355)
\curveto(341.49581702,126.79908966)(339.88628729,105.21396235)(337.90083175,105.21396235)
\curveto(335.91537621,105.21396235)(334.30584647,126.79908966)(334.30584647,153.42569355)
\curveto(334.30584647,180.05229744)(335.91537621,201.63742475)(337.90083175,201.63742475)
\curveto(339.88628729,201.63742475)(341.49581702,180.05229744)(341.49581702,153.42569355)
\closepath
}
}
{
\newrgbcolor{curcolor}{0 0 0}
\pscustom[linewidth=0.75590547,linecolor=curcolor]
{
\newpath
\moveto(341.49581702,153.42569355)
\curveto(341.49581702,126.79908966)(339.88628729,105.21396235)(337.90083175,105.21396235)
\curveto(335.91537621,105.21396235)(334.30584647,126.79908966)(334.30584647,153.42569355)
\curveto(334.30584647,180.05229744)(335.91537621,201.63742475)(337.90083175,201.63742475)
\curveto(339.88628729,201.63742475)(341.49581702,180.05229744)(341.49581702,153.42569355)
\closepath
}
}
{
\newrgbcolor{curcolor}{0 0 0}
\pscustom[linestyle=none,fillstyle=solid,fillcolor=curcolor]
{
\newpath
\moveto(78.77228213,254.61039018)
\lineto(74.84439389,254.61039018)
\lineto(74.84439389,251.52345709)
\lineto(73.64800958,251.52345709)
\lineto(73.64800958,254.61039018)
\lineto(69.72012133,254.61039018)
\lineto(69.72012133,255.51664576)
\lineto(73.64800958,255.51664576)
\lineto(73.64800958,258.60357885)
\lineto(74.84439389,258.60357885)
\lineto(74.84439389,255.51664576)
\lineto(78.77228213,255.51664576)
\closepath
}
}
{
\newrgbcolor{curcolor}{0 0 0}
\pscustom[linestyle=none,fillstyle=solid,fillcolor=curcolor]
{
\newpath
\moveto(231.35734332,255.10345)
\lineto(226.24171323,255.10345)
\lineto(226.24171323,256.03373377)
\lineto(231.35734332,256.03373377)
\closepath
}
}
{
\newrgbcolor{curcolor}{0 0 0}
\pscustom[linestyle=none,fillstyle=solid,fillcolor=curcolor]
{
\newpath
\moveto(230.85226168,155.09834765)
\lineto(225.7366316,155.09834765)
\lineto(225.7366316,156.02863142)
\lineto(230.85226168,156.02863142)
\closepath
}
}
{
\newrgbcolor{curcolor}{0 0 0}
\pscustom[linestyle=none,fillstyle=solid,fillcolor=curcolor]
{
\newpath
\moveto(361.66705714,253.08314473)
\lineto(356.55142705,253.08314473)
\lineto(356.55142705,254.0134285)
\lineto(361.66705714,254.0134285)
\closepath
}
}
{
\newrgbcolor{curcolor}{0 0 0}
\pscustom[linestyle=none,fillstyle=solid,fillcolor=curcolor]
{
\newpath
\moveto(543.94754979,154.60528594)
\lineto(540.01966154,154.60528594)
\lineto(540.01966154,151.51835285)
\lineto(538.82327723,151.51835285)
\lineto(538.82327723,154.60528594)
\lineto(534.89538899,154.60528594)
\lineto(534.89538899,155.51154152)
\lineto(538.82327723,155.51154152)
\lineto(538.82327723,158.59847461)
\lineto(540.01966154,158.59847461)
\lineto(540.01966154,155.51154152)
\lineto(543.94754979,155.51154152)
\closepath
}
}
\end{pspicture}

%% file: figure1.tex
\psset{xunit=.5pt,yunit=.5pt,runit=.5pt}
\begin{pspicture}(599.77825495,189.98260738)
{
\newrgbcolor{curcolor}{0.94901961 0.94901961 0.94901961}
\pscustom[linestyle=none,fillstyle=solid,fillcolor=curcolor]
{
\newpath
\moveto(598.27383368,143.65798862)
\curveto(598.27383368,118.95787917)(532.99355158,98.93448418)(452.46624271,98.93448418)
\curveto(371.93893384,98.93448418)(306.65865174,118.95787917)(306.65865174,143.65798862)
\curveto(306.65865174,168.35809807)(371.93893384,188.38149306)(452.46624271,188.38149306)
\curveto(532.99355158,188.38149306)(598.27383368,168.35809807)(598.27383368,143.65798862)
\closepath
}
}
{
\newrgbcolor{curcolor}{0.7019608 0.7019608 0.7019608}
\pscustom[linestyle=none,fillstyle=solid,fillcolor=curcolor]
{
\newpath
\moveto(455.4503807,144.12878625)
\curveto(455.4503807,118.91314986)(454.84307671,98.47183771)(454.09392928,98.47183771)
\curveto(453.34478185,98.47183771)(452.73747786,118.91314986)(452.73747786,144.12878625)
\curveto(452.73747786,169.34442265)(453.34478185,189.78573479)(454.09392928,189.78573479)
\curveto(454.84307671,189.78573479)(455.4503807,169.34442265)(455.4503807,144.12878625)
\closepath
}
}
{
\newrgbcolor{curcolor}{0.89019608 0.87058824 0.85882354}
\pscustom[linestyle=none,fillstyle=solid,fillcolor=curcolor]
{
\newpath
\moveto(598.27383368,144.59946893)
\curveto(598.27383368,119.89935948)(532.99355158,99.87596449)(452.46624271,99.87596449)
\curveto(371.93893384,99.87596449)(306.65865174,119.89935948)(306.65865174,144.59946893)
\curveto(306.65865174,169.29957839)(371.93893384,189.32297338)(452.46624271,189.32297338)
\curveto(532.99355158,189.32297338)(598.27383368,169.29957839)(598.27383368,144.59946893)
\closepath
}
}
{
\newrgbcolor{curcolor}{0 0 0}
\pscustom[linewidth=1.00157476,linecolor=curcolor]
{
\newpath
\moveto(598.27383368,144.59946893)
\curveto(598.27383368,119.89935948)(532.99355158,99.87596449)(452.46624271,99.87596449)
\curveto(371.93893384,99.87596449)(306.65865174,119.89935948)(306.65865174,144.59946893)
\curveto(306.65865174,169.29957839)(371.93893384,189.32297338)(452.46624271,189.32297338)
\curveto(532.99355158,189.32297338)(598.27383368,169.29957839)(598.27383368,144.59946893)
\closepath
}
}
{
\newrgbcolor{curcolor}{0.94901961 0.94901961 0.94901961}
\pscustom[linestyle=none,fillstyle=solid,fillcolor=curcolor]
{
\newpath
\moveto(291.61518298,142.16038862)
\curveto(291.61518298,117.46027917)(226.33490088,97.43688418)(145.807592,97.43688418)
\curveto(65.28028313,97.43688418)(0.00000103,117.46027917)(0.00000103,142.16038862)
\curveto(0.00000103,166.86049807)(65.28028313,186.88389306)(145.807592,186.88389306)
\curveto(226.33490088,186.88389306)(291.61518298,166.86049807)(291.61518298,142.16038862)
\closepath
}
}
{
\newrgbcolor{curcolor}{0.7019608 0.7019608 0.7019608}
\pscustom[linestyle=none,fillstyle=solid,fillcolor=curcolor]
{
\newpath
\moveto(148.79168545,142.63118537)
\curveto(148.79168545,117.41554897)(148.18438146,96.97423683)(147.43523403,96.97423683)
\curveto(146.6860866,96.97423683)(146.07878261,117.41554897)(146.07878261,142.63118537)
\curveto(146.07878261,167.84682176)(146.6860866,188.28813391)(147.43523403,188.28813391)
\curveto(148.18438146,188.28813391)(148.79168545,167.84682176)(148.79168545,142.63118537)
\closepath
}
}
{
\newrgbcolor{curcolor}{0.89019608 0.87058824 0.85882354}
\pscustom[linestyle=none,fillstyle=solid,fillcolor=curcolor]
{
\newpath
\moveto(292.7003118,144.04316027)
\curveto(292.7003118,119.34305082)(227.42002969,99.31965583)(146.89272082,99.31965583)
\curveto(66.36541195,99.31965583)(1.08512985,119.34305082)(1.08512985,144.04316027)
\curveto(1.08512985,168.74326972)(66.36541195,188.76666471)(146.89272082,188.76666471)
\curveto(227.42002969,188.76666471)(292.7003118,168.74326972)(292.7003118,144.04316027)
\closepath
}
}
{
\newrgbcolor{curcolor}{0 0 0}
\pscustom[linewidth=1.00157476,linecolor=curcolor]
{
\newpath
\moveto(292.7003118,144.04316027)
\curveto(292.7003118,119.34305082)(227.42002969,99.31965583)(146.89272082,99.31965583)
\curveto(66.36541195,99.31965583)(1.08512985,119.34305082)(1.08512985,144.04316027)
\curveto(1.08512985,168.74326972)(66.36541195,188.76666471)(146.89272082,188.76666471)
\curveto(227.42002969,188.76666471)(292.7003118,168.74326972)(292.7003118,144.04316027)
\closepath
}
}
{
\newrgbcolor{curcolor}{0.89019608 0.87058824 0.85882354}
\pscustom[linestyle=none,fillstyle=solid,fillcolor=curcolor]
{
\newpath
\moveto(599.39805575,45.38314003)
\curveto(599.39805575,20.68303058)(534.11777365,0.65963559)(453.59046478,0.65963559)
\curveto(373.0631559,0.65963559)(307.7828738,20.68303058)(307.7828738,45.38314003)
\curveto(307.7828738,70.08324949)(373.0631559,90.10664447)(453.59046478,90.10664447)
\curveto(534.11777365,90.10664447)(599.39805575,70.08324949)(599.39805575,45.38314003)
\closepath
}
}
{
\newrgbcolor{curcolor}{0 0 0}
\pscustom[linewidth=1.00157476,linecolor=curcolor]
{
\newpath
\moveto(599.39805575,45.38314003)
\curveto(599.39805575,20.68303058)(534.11777365,0.65963559)(453.59046478,0.65963559)
\curveto(373.0631559,0.65963559)(307.7828738,20.68303058)(307.7828738,45.38314003)
\curveto(307.7828738,70.08324949)(373.0631559,90.10664447)(453.59046478,90.10664447)
\curveto(534.11777365,90.10664447)(599.39805575,70.08324949)(599.39805575,45.38314003)
\closepath
}
}
{
\newrgbcolor{curcolor}{0.94901961 0.94901961 0.94901961}
\pscustom[linestyle=none,fillstyle=solid,fillcolor=curcolor]
{
\newpath
\moveto(450.33497955,62.32799052)
\curveto(450.33497955,71.89885264)(450.33497955,81.469914)(449.06900555,86.2550026)
\curveto(447.80303156,91.0400912)(445.27095437,91.0400912)(444.00490286,84.99964685)
\curveto(442.73885135,78.95920251)(442.73885135,66.72150681)(444.99965109,60.75922624)
\curveto(447.26045084,54.79694568)(451.78188956,54.79694568)(454.63043204,61.15135425)
\curveto(457.47897451,67.50576283)(458.65455128,80.21424065)(462.36222326,85.7056033)
\curveto(466.06989525,91.19696595)(472.30960816,89.47109107)(475.47460482,73.93840678)
\curveto(478.63960147,58.40572248)(478.73003643,29.06452121)(476.60497705,14.16001361)
\curveto(474.47991766,-0.74449398)(470.13929413,-1.21518454)(466.16029918,4.04087634)
\curveto(462.18130422,9.29693722)(458.5641011,20.27970478)(454.76610277,25.77106403)
\curveto(450.96810444,31.26242328)(446.98916066,31.26242535)(444.95444607,26.16445224)
\curveto(442.91973149,21.06647912)(442.82928216,10.8660193)(443.91450871,5.76679242)
\curveto(444.99973525,0.66756554)(447.26045084,0.66756554)(448.39077975,4.5114541)
\curveto(449.52110866,8.35534267)(449.52110866,16.04360546)(449.52110866,23.73138259)
}
}
{
\newrgbcolor{curcolor}{0 0 0}
\pscustom[linewidth=0.7559055,linecolor=curcolor]
{
\newpath
\moveto(450.33497955,62.32799052)
\curveto(450.33497955,71.89885264)(450.33497955,81.469914)(449.06900555,86.2550026)
\curveto(447.80303156,91.0400912)(445.27095437,91.0400912)(444.00490286,84.99964685)
\curveto(442.73885135,78.95920251)(442.73885135,66.72150681)(444.99965109,60.75922624)
\curveto(447.26045084,54.79694568)(451.78188956,54.79694568)(454.63043204,61.15135425)
\curveto(457.47897451,67.50576283)(458.65455128,80.21424065)(462.36222326,85.7056033)
\curveto(466.06989525,91.19696595)(472.30960816,89.47109107)(475.47460482,73.93840678)
\curveto(478.63960147,58.40572248)(478.73003643,29.06452121)(476.60497705,14.16001361)
\curveto(474.47991766,-0.74449398)(470.13929413,-1.21518454)(466.16029918,4.04087634)
\curveto(462.18130422,9.29693722)(458.5641011,20.27970478)(454.76610277,25.77106403)
\curveto(450.96810444,31.26242328)(446.98916066,31.26242535)(444.95444607,26.16445224)
\curveto(442.91973149,21.06647912)(442.82928216,10.8660193)(443.91450871,5.76679242)
\curveto(444.99973525,0.66756554)(447.26045084,0.66756554)(448.39077975,4.5114541)
\curveto(449.52110866,8.35534267)(449.52110866,16.04360546)(449.52110866,23.73138259)
}
}
{
\newrgbcolor{curcolor}{0.89019608 0.87058824 0.85882354}
\pscustom[linestyle=none,fillstyle=solid,fillcolor=curcolor]
{
\newpath
\moveto(295.31902709,45.35373682)
\curveto(295.31902709,20.65362737)(230.03874499,0.63023238)(149.51143612,0.63023238)
\curveto(68.98412724,0.63023238)(3.70384514,20.65362737)(3.70384514,45.35373682)
\curveto(3.70384514,70.05384627)(68.98412724,90.07724126)(149.51143612,90.07724126)
\curveto(230.03874499,90.07724126)(295.31902709,70.05384627)(295.31902709,45.35373682)
\closepath
}
}
{
\newrgbcolor{curcolor}{0 0 0}
\pscustom[linewidth=0.7559055,linecolor=curcolor]
{
\newpath
\moveto(295.31902709,45.35373682)
\curveto(295.31902709,20.65362737)(230.03874499,0.63023238)(149.51143612,0.63023238)
\curveto(68.98412724,0.63023238)(3.70384514,20.65362737)(3.70384514,45.35373682)
\curveto(3.70384514,70.05384627)(68.98412724,90.07724126)(149.51143612,90.07724126)
\curveto(230.03874499,90.07724126)(295.31902709,70.05384627)(295.31902709,45.35373682)
\closepath
}
}
{
\newrgbcolor{curcolor}{0.94901961 0.94901961 0.94901961}
\pscustom[linestyle=none,fillstyle=solid,fillcolor=curcolor]
{
\newpath
\moveto(146.25595089,62.29858731)
\curveto(146.25595089,71.86944942)(146.25595089,81.44051078)(144.98997689,86.22559939)
\curveto(143.72400289,91.01068799)(141.19192571,91.01068799)(139.9258742,84.97024364)
\curveto(138.65982268,78.92979929)(138.65982268,66.69210359)(140.92062243,60.72982303)
\curveto(143.18142218,54.76754246)(147.7028609,54.76754246)(150.55140338,61.12195104)
\curveto(153.39994585,67.47635962)(154.57552261,80.18483744)(158.2831946,85.67620009)
\curveto(161.99086659,91.16756273)(168.2305795,89.44168786)(171.39557615,73.90900356)
\curveto(174.56057281,58.37631927)(174.65100777,29.035118)(172.52594838,14.1306104)
\curveto(170.400889,-0.7738972)(166.06026547,-1.24458775)(162.08127052,4.01147313)
\curveto(158.10227556,9.267534)(154.48507244,20.25030157)(150.68707411,25.74166082)
\curveto(146.88907578,31.23302007)(142.910132,31.23302214)(140.87541741,26.13504902)
\curveto(138.84070282,21.03707591)(138.7502535,10.83661608)(139.83548004,5.7373892)
\curveto(140.92070659,0.63816232)(143.18142218,0.63816232)(144.31175109,4.48205089)
\curveto(145.44207999,8.32593945)(145.44207999,16.01420225)(145.44207999,23.70197937)
}
}
{
\newrgbcolor{curcolor}{0 0 0}
\pscustom[linewidth=0.7559055,linecolor=curcolor]
{
\newpath
\moveto(146.25595089,62.29858731)
\curveto(146.25595089,71.86944942)(146.25595089,81.44051078)(144.98997689,86.22559939)
\curveto(143.72400289,91.01068799)(141.19192571,91.01068799)(139.9258742,84.97024364)
\curveto(138.65982268,78.92979929)(138.65982268,66.69210359)(140.92062243,60.72982303)
\curveto(143.18142218,54.76754246)(147.7028609,54.76754246)(150.55140338,61.12195104)
\curveto(153.39994585,67.47635962)(154.57552261,80.18483744)(158.2831946,85.67620009)
\curveto(161.99086659,91.16756273)(168.2305795,89.44168786)(171.39557615,73.90900356)
\curveto(174.56057281,58.37631927)(174.65100777,29.035118)(172.52594838,14.1306104)
\curveto(170.400889,-0.7738972)(166.06026547,-1.24458775)(162.08127052,4.01147313)
\curveto(158.10227556,9.267534)(154.48507244,20.25030157)(150.68707411,25.74166082)
\curveto(146.88907578,31.23302007)(142.910132,31.23302214)(140.87541741,26.13504902)
\curveto(138.84070282,21.03707591)(138.7502535,10.83661608)(139.83548004,5.7373892)
\curveto(140.92070659,0.63816232)(143.18142218,0.63816232)(144.31175109,4.48205089)
\curveto(145.44207999,8.32593945)(145.44207999,16.01420225)(145.44207999,23.70197937)
}
}
{
\newrgbcolor{curcolor}{0.94901961 0.94901961 0.94901961}
\pscustom[linestyle=none,fillstyle=solid,fillcolor=curcolor]
{
\newpath
\moveto(155.04125633,144.02636253)
\curveto(155.04125633,119.48034586)(153.38149826,99.58186776)(151.33408291,99.58186776)
\curveto(149.28666757,99.58186776)(147.6269095,119.48034586)(147.6269095,144.02636253)
\curveto(147.6269095,168.5723792)(149.28666757,188.47085731)(151.33408291,188.47085731)
\curveto(153.38149826,188.47085731)(155.04125633,168.5723792)(155.04125633,144.02636253)
\closepath
}
}
{
\newrgbcolor{curcolor}{0 0 0}
\pscustom[linewidth=0.75590552,linecolor=curcolor]
{
\newpath
\moveto(155.04125633,144.02636253)
\curveto(155.04125633,119.48034586)(153.38149826,99.58186776)(151.33408291,99.58186776)
\curveto(149.28666757,99.58186776)(147.6269095,119.48034586)(147.6269095,144.02636253)
\curveto(147.6269095,168.5723792)(149.28666757,188.47085731)(151.33408291,188.47085731)
\curveto(153.38149826,188.47085731)(155.04125633,168.5723792)(155.04125633,144.02636253)
\closepath
}
}
{
\newrgbcolor{curcolor}{0 0 0}
\pscustom[linestyle=none,fillstyle=solid,fillcolor=curcolor]
{
\newpath
\moveto(460.85943101,8.6002785)
\lineto(456.93154277,8.6002785)
\lineto(456.93154277,5.51334542)
\lineto(455.73515846,5.51334542)
\lineto(455.73515846,8.6002785)
\lineto(451.80727022,8.6002785)
\lineto(451.80727022,9.50653409)
\lineto(455.73515846,9.50653409)
\lineto(455.73515846,12.59346717)
\lineto(456.93154277,12.59346717)
\lineto(456.93154277,9.50653409)
\lineto(460.85943101,9.50653409)
\closepath
}
}
{
\newrgbcolor{curcolor}{0 0 0}
\pscustom[linestyle=none,fillstyle=solid,fillcolor=curcolor]
{
\newpath
\moveto(395.47824712,145.14834656)
\lineto(391.55035888,145.14834656)
\lineto(391.55035888,142.06141348)
\lineto(390.35397457,142.06141348)
\lineto(390.35397457,145.14834656)
\lineto(386.42608633,145.14834656)
\lineto(386.42608633,146.05460214)
\lineto(390.35397457,146.05460214)
\lineto(390.35397457,149.14153523)
\lineto(391.55035888,149.14153523)
\lineto(391.55035888,146.05460214)
\lineto(395.47824712,146.05460214)
\closepath
}
}
{
\newrgbcolor{curcolor}{0 0 0}
\pscustom[linestyle=none,fillstyle=solid,fillcolor=curcolor]
{
\newpath
\moveto(209.44567003,145.67760269)
\lineto(205.51778179,145.67760269)
\lineto(205.51778179,142.5906696)
\lineto(204.32139748,142.5906696)
\lineto(204.32139748,145.67760269)
\lineto(200.39350924,145.67760269)
\lineto(200.39350924,146.58385827)
\lineto(204.32139748,146.58385827)
\lineto(204.32139748,149.67079135)
\lineto(205.51778179,149.67079135)
\lineto(205.51778179,146.58385827)
\lineto(209.44567003,146.58385827)
\closepath
}
}
{
\newrgbcolor{curcolor}{0 0 0}
\pscustom[linestyle=none,fillstyle=solid,fillcolor=curcolor]
{
\newpath
\moveto(98.25919111,146.67648934)
\lineto(94.5418866,146.67648934)
\lineto(94.5418866,147.70169078)
\lineto(98.25919111,147.70169078)
\closepath
}
}
{
\newrgbcolor{curcolor}{0 0 0}
\pscustom[linestyle=none,fillstyle=solid,fillcolor=curcolor]
{
\newpath
\moveto(523.57300894,145.61793527)
\lineto(519.85570443,145.61793527)
\lineto(519.85570443,146.64313672)
\lineto(523.57300894,146.64313672)
\closepath
}
}
{
\newrgbcolor{curcolor}{0 0 0}
\pscustom[linestyle=none,fillstyle=solid,fillcolor=curcolor]
{
\newpath
\moveto(153.52986294,10.65791533)
\lineto(149.81255842,10.65791533)
\lineto(149.81255842,11.68311677)
\lineto(153.52986294,11.68311677)
\closepath
}
}
{
\newrgbcolor{curcolor}{0 0 0}
\pscustom[linestyle=none,fillstyle=solid,fillcolor=curcolor]
{
\newpath
\moveto(213.48984518,46.70670688)
\lineto(209.56195694,46.70670688)
\lineto(209.56195694,43.6197738)
\lineto(208.36557263,43.6197738)
\lineto(208.36557263,46.70670688)
\lineto(204.43768439,46.70670688)
\lineto(204.43768439,47.61296247)
\lineto(208.36557263,47.61296247)
\lineto(208.36557263,50.69989555)
\lineto(209.56195694,50.69989555)
\lineto(209.56195694,47.61296247)
\lineto(213.48984518,47.61296247)
\closepath
}
}
{
\newrgbcolor{curcolor}{0 0 0}
\pscustom[linewidth=1.00157475,linecolor=curcolor,linestyle=dashed,dash=0.26499999 0.26499999]
{
\newpath
\moveto(145.44207496,23.7019825)
\lineto(146.25596598,62.29857085)
\lineto(146.25596598,62.29857085)
}
}
{
\newrgbcolor{curcolor}{0.94901961 0.94901961 0.94901961}
\pscustom[linestyle=none,fillstyle=solid,fillcolor=curcolor]
{
\newpath
\moveto(449.52109228,23.73137967)
\lineto(450.33497575,62.32799069)
}
}
{
\newrgbcolor{curcolor}{0 0 0}
\pscustom[linewidth=0.75590552,linecolor=curcolor,linestyle=dashed,dash=0.2 0.2]
{
\newpath
\moveto(449.52109228,23.73137967)
\lineto(450.33497575,62.32799069)
}
}
{
\newrgbcolor{curcolor}{0 0 0}
\pscustom[linestyle=none,fillstyle=solid,fillcolor=curcolor]
{
\newpath
\moveto(97.38321711,49.54687001)
\lineto(93.45532887,49.54687001)
\lineto(93.45532887,46.45993693)
\lineto(92.25894456,46.45993693)
\lineto(92.25894456,49.54687001)
\lineto(88.33105632,49.54687001)
\lineto(88.33105632,50.4531256)
\lineto(92.25894456,50.4531256)
\lineto(92.25894456,53.54005868)
\lineto(93.45532887,53.54005868)
\lineto(93.45532887,50.4531256)
\lineto(97.38321711,50.4531256)
\closepath
}
}
{
\newrgbcolor{curcolor}{0 0 0}
\pscustom[linestyle=none,fillstyle=solid,fillcolor=curcolor]
{
\newpath
\moveto(394.00152247,48.77311591)
\lineto(390.28421796,48.77311591)
\lineto(390.28421796,49.79831736)
\lineto(394.00152247,49.79831736)
\closepath
}
}
{
\newrgbcolor{curcolor}{0 0 0}
\pscustom[linestyle=none,fillstyle=solid,fillcolor=curcolor]
{
\newpath
\moveto(460.24037781,83.83257735)
\lineto(456.31248957,83.83257735)
\lineto(456.31248957,80.74564427)
\lineto(455.11610526,80.74564427)
\lineto(455.11610526,83.83257735)
\lineto(451.18821702,83.83257735)
\lineto(451.18821702,84.73883293)
\lineto(455.11610526,84.73883293)
\lineto(455.11610526,87.82576601)
\lineto(456.31248957,87.82576601)
\lineto(456.31248957,84.73883293)
\lineto(460.24037781,84.73883293)
\closepath
}
}
{
\newrgbcolor{curcolor}{0.94901961 0.94901961 0.94901961}
\pscustom[linestyle=none,fillstyle=solid,fillcolor=curcolor]
{
\newpath
\moveto(460.13574067,145.00000063)
\curveto(460.13574067,120.45398396)(458.4759826,100.55550586)(456.42856726,100.55550586)
\curveto(454.38115192,100.55550586)(452.72139384,120.45398396)(452.72139384,145.00000063)
\curveto(452.72139384,169.5460173)(454.38115192,189.4444954)(456.42856726,189.4444954)
\curveto(458.4759826,189.4444954)(460.13574067,169.5460173)(460.13574067,145.00000063)
\closepath
}
}
{
\newrgbcolor{curcolor}{0 0 0}
\pscustom[linewidth=0.75590552,linecolor=curcolor]
{
\newpath
\moveto(460.13574067,145.00000063)
\curveto(460.13574067,120.45398396)(458.4759826,100.55550586)(456.42856726,100.55550586)
\curveto(454.38115192,100.55550586)(452.72139384,120.45398396)(452.72139384,145.00000063)
\curveto(452.72139384,169.5460173)(454.38115192,189.4444954)(456.42856726,189.4444954)
\curveto(458.4759826,189.4444954)(460.13574067,169.5460173)(460.13574067,145.00000063)
\closepath
}
}
{
\newrgbcolor{curcolor}{0 0 0}
\pscustom[linestyle=none,fillstyle=solid,fillcolor=curcolor]
{
\newpath
\moveto(153.64436661,80.91596897)
\lineto(149.92706209,80.91596897)
\lineto(149.92706209,81.94117042)
\lineto(153.64436661,81.94117042)
\closepath
}
}
{
\newrgbcolor{curcolor}{0 0 0}
\pscustom[linestyle=none,fillstyle=solid,fillcolor=curcolor]
{
\newpath
\moveto(523.6443963,48.41597201)
\lineto(519.92709178,48.41597201)
\lineto(519.92709178,49.44117346)
\lineto(523.6443963,49.44117346)
\closepath
}
}
\end{pspicture}